\documentclass[12pt]{article}
\usepackage{amsmath}
\usepackage{amssymb}
\usepackage{amsthm}
\usepackage[english]{babel}
\usepackage{nccmath}
\newtheorem{theorem}{Theorem}
\newtheorem{lemma}{Lemma}
\newtheorem{remark}{Remark}

\begin{document}
\begin{center}
\section*{Interpolation theorem for anisotropic net spaces}
\bf{A.N.Bashirova, A.H.Kalidolday, E.D.Nursultanov}
\end{center} 

{\bf Annotation.} The paper studies the interpolation properties of anisotropic net spaces $N_{\bar{p},\bar{q}}(M)$, where $\bar{p}=(p_1, p_2)$, $\bar{q}=(q_1, q_2)$. 
It is shown that the following equality holds with respect to the multidimensional interpolation method 
$$
(N_{\bar{p}_0,\bar{q}_0}(M), N_{\bar{p}_1,\bar{q}_1}(M))_{\bar{\theta},\bar{q}}=N_{\bar{p},\bar{q}}(M),\;\;\;
\frac{1}{\bar{p}}=\frac{1-\bar{\theta}}{\bar{p}_0}+\frac{\bar{\theta}}{\bar{p}_1}.
$$

{\bf Keywords:} net spaces, anisotropic spaces, real interpolation method.

\section{Introduction}

Let $M$ is the set of all segments from $\mathbb{R}$. For function $f(x)$, defined and integrable on each segment $Q$ from $M$, define the function
$$
\bar{f}(t,M)=\sup_{{}^{Q\in M}_{|Q|>t}}\frac{1}{|Q|}\left|\int_Q f(x)dx\right|,\;\;\;t>0,
$$
where the exact upper edge is taken over all segments $Q\in M$, whose length $|Q|>t$. Function $\bar{f}(t,M)$ is called averaging the function $f$ over the net $M$.

By $N_{p,q}(M)$, $0<p,q\leq\infty$ we denote the set of functions $f$, for which when $q<\infty$
$$
\left\|f\right\|_{N_{p,q}(M)}= \left(\int^{\infty}_{0}\left(t^{\frac{1}{p}}\bar{f}(t,M)\right)^{q}\frac{dt}{t}\right)^\frac{1}{q}<\infty
$$
and when $q=\infty$
$$
\left\|f\right\|_{N_{p,\infty}(M)}=\sup_{t>0}t^\frac{1}{p}\bar{f}(t,M)<\infty.
$$

These spaces are called net spaces, they were introduced and studied in the work \cite{Nursultanov1998Net}. Net spaces are an important research tool in the theory of Fourier series, in operator theory and in other directions. \cite{NED1}-\cite{ARN3}.

In the paper   \cite{Nursultanov2} it was shown that this scale of spaces $N_{p,q}(M)$ is closed with respect to the real interpolation method, that is, at $p_0\neq p_1$ takes place
$$
\left(N_{p_0,q_0}(M), N_{p_1,q_1}(M)\right)_{\theta,q}=N_{p,q}(M).
$$

If in the definition of the space $N_{p,q}(M)$ instead of $\bar{f}(t,M)$ consider the function 
$$
\sup_{{}^{Q\in M}_{|Q|>t}}\frac{1}{|Q|}\int_Q |f(x)|dx,
$$ 
then the corresponding space, as can be seen from the work of \cite{NursultanovBur}, coincides with the Morrey space $M_{p,q}^\alpha$, where $\alpha=\frac{1}{p}-\frac{1}{q}$, but for these spaces it is known that they are not closed with respect to the real interpolation method (see \cite{P1}-\cite{BRV}).

Consider a generalization of the space $N_{p,q}(M)$ in two-dimensional case. 

Let $M$ is set of all rectangles $Q=Q_1\times Q_2$ from $\mathbb{R}^2$, for function $f(x_1,x_2)$ integrable on each set $Q\in M$ define 
$$
\displaystyle\bar{f}(t_1,t_2; M)=\sup_{|Q_i|\geq t_i}\frac{1}{|Q_1||Q_2|}\left|\int_{Q_1}\int_{Q_2}{f(x_1,x_2)dx_1dx_2}\right|,\;\;\;\;t_i>0,
$$
where $|Q_i|$ is the $Q_i$ segment length.

Let $0<\bar{p}=(p_1,p_2)<\infty$, $0<\bar{q}=(q_1,q_2)\leq\infty$. By $N_{\bar{p},\bar{q}}(M)$ denote the set of all functions $f(x_1,x_2)$, for which
$$
\left\|f\right\|_{N_{\bar{p},\bar{q}}(M)}= \left(\int^{\infty}_{0}\left(\int^{\infty}_{0}\left(t_1^{\frac{1}{p_1}}t_2^{\frac{1}{p_2}}\bar{f}(t_1,t_2; M)\right)^{q_1}\frac{dt_1}{t_1}\right)^{\frac{q_2}{q_1}}\frac{dt_2}{t_2}\right)^\frac{1}{q_2}<\infty,
$$
hereinafter, when $q=\infty$, expression $\left(\int^\infty_0\left(\varphi(t)\right)^q\frac{dt}{t}\right)^\frac{1}{q}$ is understood as $\sup_{t>0}\varphi(t)$.

As it can be seen from the definition of the space $N_{\bar{p},\bar{q}}(M)$, this is the space of functions that have different characteristics for each variable. These spaces are called anisotropic net spaces.  

For spaces with mixed metrics, anisotropic spaces, the real interpolation method does not work. For the interpolation of spaces with mixed metric, the interpolation method was introduced by D.L. Fernandez \cite{Fernandez1}-\cite{Fernandez3} and his modification \cite{NED4}, \cite{Nursultanov2}, \cite{NED6}. An interpolation theorem regarding this method for the Lebesgue spaces $L_{\bar{p}}$ with mixed metric was obtained in the work \cite{NED6}: {\it let $0<\bar{p}_i<\infty$ and $p_0^i\neq p_1^i$, $i=0,1$, $0<\bar{q}\leq\infty$, $0<\bar{\theta}<1$, then
$$
(L_{\bar{p}_0}, L_{\bar{p}_1})_{\bar{\theta},\bar{q}}=L_{\bar{p},\bar{q}},\;\;\;\frac{1}{\bar{p}}=\frac{1-\bar{\theta}}{\bar{p}_0}+\frac{\bar{\theta}}{\bar{p}_1},
$$
where $L_{\bar{p},\bar{q}}$ - anisotropic Lorentz space.} (see \cite{Blozinski})

Other applications of this method can be found in the works \cite{Nursultanov5}, \cite{Nursultanov6}. 

The purpose of this paper is to obtain an interpolation theorem for anisotropic net spaces. 

\section{Lemmas}

\begin{lemma}\label{l2}
Let $\varphi(x)$ is locally integrable function, $\displaystyle\mathbb{R}=\bigcup_{k}I_k$ partition of $\mathbb{R}$ into segments of length $|I_k|=\tau>0$, $k\in\mathbb{Z}$, at that $|I_k\cap I_j|=0$, $k\neq j$.
Then for an arbitrary segment $Q$ of length $|Q|\geq\tau$ there are segments $Q'$, $Q''$, $Q'''$ such 
that  $Q'''$ is consists of the union of an integer number of segments of the partition and $\tau\leq|Q'|\leq 2\tau$, $\tau\leq|Q''|\leq2\tau$, $|Q|\leq|Q'''|\leq3|Q|$ and the inequality holds
$$
\left|\int_{Q}\varphi(x)dx\right|\leq\left|\int_{Q'}\varphi(x)dx\right|
+\left|\int_{Q''}\varphi(x)dx\right|+\left|\int_{Q'''}\varphi(x)dx\right|.
$$
\end{lemma}
\begin{proof}
Let the function $\varphi$ satisfies the conditions of the lemma, $Q\subset\mathbb{R}$ is segment and $|Q|\geq\tau$. Let 
$$
\bigcup_{I_k\cap Q\neq\oslash}I_k=\bigcup_{k=k'}^{k''}I_k.
$$
Let's take
$$
Q'''=\bigcup_{k=k'-1}^{k''+1}I_k,
$$
$$
Q'=\left(I_{k'-1}\cup I_{k'}\right)\setminus Q,
$$ 
$$
Q''=\left(I_{k''}\cup I_{k''+1}\right)\setminus Q.
$$

Then we have $\tau\leq|Q'|\leq 2\tau$, $\tau\leq|Q''|\leq2\tau$, $|Q|\leq|Q'''|\leq3|Q|$ and
\begin{equation*}
\begin{split}
\left|\int_{Q}\varphi(x)dx\right|=\left|\int_{Q'''}\varphi(x)dx-\int_{Q'}\varphi(x)dx-\int_{Q''}\varphi(x)dx\right|\leq \\
\leq\left|\int_{Q'''}\varphi(x)dx\right|+\left|\int_{Q'}\varphi(x)dx\right|+\left|\int_{Q''}\varphi(x)dx\right|.
\end{split}
\end{equation*}

\end{proof}

\begin{lemma}\label{l1}
Let $\displaystyle\mathbb{R}=\bigcup_{k}I_k$ partition of $\mathbb{R}$ 
into segments of length $|I_k|=\tau>0$, $k\in\mathbb{Z}$, at that $|I_k\cap I_j|=0$, $k\neq j$. Let $\varphi$be a function such that
\begin{equation}\label{fi}
\int_{I_k}\varphi(x)dx=0, \;\;\; k\in\mathbb{Z},
\end{equation} 
then for an arbitrary segment $Q$ of length $|Q|\geq\tau$ there are segments $Q'$ and $Q''$ such that, 
 $\tau\leq|Q'|\leq2\tau$, $\tau\leq|Q''|\leq 2\tau$ and
$$
\left|\int_{Q}\varphi(x)dx\right|\leq\left|\int_{Q'}\varphi(x)dx\right|
+\left|\int_{Q''}\varphi(x)dx\right|.
$$
\end{lemma}

The proof follows immediately from the lemma \ref{l2} and the ratio \eqref{fi}.

Let $\tau=(\tau_1,\tau_2)\in(0,+\infty)^2$, $I_k^1=[0,\tau_1]+k\tau_1$, $k\in \mathbb{Z}$, $I_m^2=[0,\tau_2]+m\tau_2$, $m\in\mathbb{Z}$. The system sets $G_\tau=\left\{I_{km}=I_k\times I_m\right\}_{(k,m)\in\mathbb{Z}^2}$ gives a partition of $\mathbb{R}^2$ into rectangles, i.e. $\displaystyle\mathbb{R}^2=\bigcup_{k,m\in \mathbb{Z}^2} I_{k m}$. 

For a locally integrable function $f(x_1,x_2)$ and the set $G_\tau$ we define the functions $f_{00}(x_1,x_2)$, $f_{01}(x_1,x_2)$, $f_{10}(x_1,x_2)$, $f_{11}(x_1,x_2)$ in the following way: 
\begin{equation}\label{f01}
f_{01}(x_1,x_2)=\frac{1}{|I_m^2|}\int_{I_m^2}f({x_1},{x_2}'){dx_2}'-\frac{1}{|I_k^1||I_m^2|}\int_{I_k^1}\int_{I_m^2}f({x_1}',{x_2}'){dx_1}'{dx_2}', \;\;\;\; 
(x_1,x_2)\in{I_k^1}\times{I_m^2},
\end{equation}
\begin{equation}\label{f10}
f_{10}(x_1,x_2)=\frac{1}{|I_k^1|}\int_{I_k^1}f({x_1}',{x_2}){dx_1}'-\frac{1}{|I_k^1||I_m^2|}\int_{I_k^1}\int_{I_m^2}f({x_1}',{x_2}'){dx_1}'{dx_2}', \;\;\;\; 
(x_1,x_2)\in{I_k^1}\times{I_m^2},
\end{equation}
\begin{equation}\label{f11}
f_{11}(x_1,x_2)=\frac1{|I_k^1||I_m^2|}\int_{I_k^1}\int_{I_m^2}f(x_1',x_2')dx_1'dx_2',\;\;\;\;
(x_1,x_2)\in I_k^1\times I_m^2,
\end{equation}
\begin{equation}\label{f00}
f_{00} = f - f_{01} - f_{10} - f_{11},
\end{equation}
i.e. 
$$
f = f_{00} + f_{10} + f_{01} + f_{11}.
$$

These functions will be called the decomposition of the function $f(x_1,x_2)$, the corresponding partition $G_\tau$.

\begin{lemma}\label{l3}
Let $G_\tau$ is partitioning $\mathbb{R}^2$ into rectangles, $f(x_1,x_2)$ is locally integrable on $\mathbb{R}^2$. $f = f_{00} + f_{10} + f_{01} + f_{11}$ is the decomposition corresponding to the partition  $G_\tau$. Then
$$
\int_{I_k^1}f_{00}(x_1,x_2)dx_1=\int_{I_k^1}f_{01}(x_1,x_2)dx_1=0, \;\;\;\; k\in\mathbb{Z},  x_2\in \mathbb{R}
$$
$$
\int_{I_m^2}f_{00}(x_1,x_2)dx_2=\int_{I_m^2}f_{10}(x_1,x_2)dx_2=0,\;\;\;\;  m\in\mathbb{Z},   x_1\in \mathbb{R}
$$
\end{lemma}

The proof immediately follows from the definitions of the functions $f_{00}$, $f_{10}$, $f_{01}$.

\begin{lemma}\label{lf00}
Let $G_\tau$ is partitioning $\mathbb{R}^2$ into rectangles, $f(x_1,x_2)$ is locally integrable on $\mathbb{R}^2$ and function $f_{00}$ defined by equality \eqref{f00}. 
Then
\begin{equation}\label{lf00_1}
\bar{f}_{00}(t_1,t_2; M)\leq\left\lbrace
\begin{aligned}
    & 64\frac{\tau_1}{t_1}\cdot\frac{\tau_2}{t_2}\bar{f}\left(\tau_1,\tau_2; M\right),\;\;\;  t_1>\tau_1, t_2>\tau_2  \\
    & 64\frac{\tau_1}{t_1}\bar{f}\left(\tau_1,t_2; M\right),\;\;\;\;\;\; t_1>\tau_1, t_2\leq\tau_2  \\
    & 64\frac{\tau_2}{t_2}\bar{f}\left(t_1,\tau_2; M\right),\;\;\;\;\;\; t_1\leq\tau_1, t_2>\tau_2  \\
    & 64\bar{f}(t_1,t_2;M),\;\;\;\;\;\;\;\;\;\;\;\;\; t_1\leq\tau_1, t_2\leq\tau_2\\
\end{aligned}
\right.
\end{equation}
\end{lemma}

\begin{proof} 
Let $Q=Q_1\times Q_2 \in M$, $|Q_1|=s_1$, $|Q_2|=s_2$. We prove the following inequality
\begin{equation}\label{lf00_2}
\frac{1}{|Q_1||Q_2|}\left|\int_{Q_2}\int_{Q_1}f_{00}(x_1,x_2)dx_1dx_2 \right|\leq\left\lbrace
\begin{aligned}
    & 64\frac{\tau_1}{s_1}\cdot\frac{\tau_2}{s_2}\bar{f}\left(\tau_1,\tau_2; M\right),\;\;\;  s_1>\tau_1, s_2>\tau_2  \\
    & 16\frac{\tau_1}{s_1}\bar{f}\left(\tau_1,s_2; M\right),\;\;\;\;\;\; s_1>\tau_1, s_2\leq\tau_2  \\
    & 16\frac{\tau_2}{s_2}\bar{f}\left(s_1,\tau_2; M\right),\;\;\;\;\;\; s_1\leq\tau_1, s_2>\tau_2  \\
    & 4\bar{f}(s_1,s_2;M),\;\;\;\;\;\;\;\;\;\;\;\;\; s_1\leq\tau_1, s_2\leq\tau_2\\
\end{aligned}
\right.
\end{equation}

Consider the case $s_1\leq\tau_1$, $s_2\leq\tau_2$. Using the definition of the function ${f}_{00}$, we get

\begin{equation}\label{prf00}
\begin{split}
&\frac{1}{|Q_1||Q_2|}\left|\int_{Q_2}\int_{Q_1}f_{00}(x_1,x_2)dx_1dx_2 \right|\leq\frac{1}{|Q_1||Q_2|}\left|\int_{Q_2}\int_{Q_1}f(x_1,x_2)dx_1dx_2\right|+\\
&+\left|\frac{1}{|Q_1|}\sum_{|{I_m^2}\cap Q_2|>0}\frac{|{I_m^2}\cap Q_2|}{|I_m^2||Q_2|}\int_{Q_1}\int_{I_m^2}f(x_1,x_2)dx_2dx_1\right|+\\
&+\left|\frac{1}{|Q_2|}\sum_{|{I_k^1}\cap Q_1|>0}\frac{|{I_k^1}\cap Q_1|}{|I_k^1||Q_1|}\int_{Q_2}\int_{I_k^1}f(x_1,x_2)dx_1dx_2\right|+\\
&+\left|\sum_{|{I_k^1}\cap Q_1|>0}\sum_{|{I_m^2}\cap Q_2|>0}\frac{|{I_k^1}\cap Q_1||{I_m^2}\cap Q_2|}{|I_k^1||I_m^2||Q_1||Q_2|}\int_{I_k^1}\int_{I_m^2}f(x_1,x_2)dx_2dx_1\right|.
\end{split}
\end{equation}

Further, we have
$$
\frac{1}{|Q_1||Q_2|}\left|\int_{Q_2}\int_{Q_1}f_{00}(x_1,x_2)dx_1dx_2 \right|\leq\frac{1}{|Q_1||Q_2|}\left|\int_{Q_2}\int_{Q_1}f(x_1,x_2)dx_1dx_2\right|+
$$
$$
+\sup_{|e_2|\geq\tau_2,e_2\in W}\frac{1}{|Q_1||e_2|}\left|\int_{e_2}\int_{Q_1}f(x_1,x_2)dx_1dx_2\right|\frac{1}{|Q_2|}\sum_{|{I_m^2}\cap Q_2|>0}|{I_m^2}\cap Q_2|+
$$
$$
+\sup_{|e_1|\geq\tau_1,e_1\in W}\frac{1}{|Q_2||e_1|}\left|\int_{Q_2}\int_{e_1}f(x_1,x_2)dx_1dx_2\right|\frac{1}{|Q_1|}\sum_{|{I_k^1}\cap Q_1|>0}|{I_k^1}\cap Q_1|+
$$
$$
+\bar{f}(\tau_1,\tau_2)\frac{1}{|Q_1||Q_2|}\sum_{|{I_k^1}\cap Q_1|>0}\sum_{|{I_m^2}\cap Q_2|>0}|{I_k^1}\cap Q_1||{I_m^2}\cap Q_2|,
$$
here $W$ - set of segments in $\mathbb{R}$.
Thus,
\begin{equation}\label{ocf00}
\begin{split}
&\frac{1}{|Q_1||Q_2|}\left|\int_{Q_2}\int_{Q_1}f_{00}(x_1,x_2)dx_1dx_2 \right|\leq\\
&\leq\bar{f}(s_1,s_2; M)+\bar{f}(s_1,\tau_2; M)+\bar{f}(\tau_1,s_2; M)+\bar{f}(\tau_1,\tau_2; M)\leq4\bar{f}(s_1,s_2; M).
\end{split}
\end{equation}

Consider the case when $|Q_1|=s_1>\tau_1$, $|Q_2|=s_2\leq\tau_2$. Taking into account the lemma  \ref{l3}, note that the function $\varphi(x)=\int_{Q_2}f_{00}(x_1,x_2)dx_2$ satisfies Lemma \ref{l1},  therefore there are segments $Q'$ and $Q''$ such that $\tau\leq|Q'|\leq2\tau$, $\tau\leq|Q''|\leq 2\tau$ and
\begin{equation*}
\begin{split}
&\frac{1}{|Q_1||Q_2|}\left|\int_{Q_2}\int_{Q_1}f_{00}(x_1,x_2)dx_1dx_2\right|\leq\frac{1}{|Q_1||Q_2|}\left(\left|\int_{Q_2}\int_{Q_1'} f_{00}(x_1,x_2)dx_1dx_2\right|+\right.\\
&\left.+\left|\int_{Q_2}\int_{Q_1''} f_{00}(x_1,x_2)dx_1dx_2\right|\right)\leq\frac{2\tau_1}{s_1}\left(\frac{1}{|Q_1'||Q_2|}\left|\int_{Q_2}\int_{Q_1'} f_{00}(x_1,x_2)dx_1dx_2\right|+\right.\\
&\left.+\frac{1}{|Q_1''||Q_2|}\left|\int_{Q_2}\int_{Q_1''} f_{00}(x_1,x_2)dx_1dx_2\right|\right).
\end{split}
\end{equation*}

Then, similarly to what was proved above (see \eqref{ocf00}), we have
\begin{equation*}
\begin{split}
&\frac{1}{|Q_1||Q_2|}\left|\int_{Q_2}\int_{Q_1}f_{00}(x_1,x_2)dx_1dx_2\right|\leq\\
&\leq\frac{2\tau_1}{s_1}\left(4\bar{f}(|{Q_1}'|,s_2; M)+4\bar{f}(|{Q_1}''|, s_2; M)\right)\leq 16\frac{\tau_1}{s_1}\bar{f}\left(\tau_1, s_2; M\right).
\end{split}
\end{equation*}

Similarly, we have in the case $|Q_1|=s_1\leq\tau_1$, $|Q_2|=s_2>\tau_2$.
\begin{equation*}
\frac{1}{|Q_1||Q_2|}\left|\int_{Q_2}\int_{Q_1}f_{00}(x_1,x_2)dx_1dx_2\right|\leq16\frac{\tau_2}{s_2}\bar{f}\left(s_1,\tau_2; M\right).
\end{equation*}

Let $|Q_1|=s_1>\tau_1$, $|Q_2|=s_2>\tau_2$. Applying lemma \ref{l1} and lemma \ref{l3}, we obtain
\begin{equation*}
\begin{split}
&\frac{1}{|Q_1||Q_2|}\left|\int_{Q_2}\int_{Q_1}f_{00}(x_1,x_2)dx_1dx_2\right|\leq\frac{1}{|Q_1||Q_2|}\left(\left|\int_{Q_2'}\int_{Q_1'} f_{00}(x_1,x_2)dx_1dx_2\right|+\right.\\
&\left.+\left|\int_{Q_2''}\int_{Q_1'} f_{00}(x_1,x_2)dx_1dx_2\right|+\left|\int_{Q_2'}\int_{Q_1''} f_{00}(x_1,x_2)dx_1dx_2\right|+\left|\int_{Q_2''}\int_{Q_1''} f_{00}(x_1,x_2)dx_1dx_2\right|\right),
\end{split}
\end{equation*}
where $\tau_i\leq|Q_i'|<2\tau_i$, $\tau_i\leq|Q_i''|<2\tau_i$, $i=1,2$.

Thus, using the estimate  \eqref{ocf00} for each term, we have 
\begin{equation*}
\frac{1}{|Q_1||Q_2|}\left|\int_{Q_2}\int_{Q_1}f_{00}(x_1,x_2)dx_1dx_2\right|\leq64\frac{\tau_1}{s_1}\frac{\tau_2}{s_2}\bar{f}\left(\tau_1,\tau_2; M\right).
\end{equation*}

Recall the definition of averaging the function $f_{00}(x_1,x_2)$ over the net $M$:
\begin{equation*}
{\bar{f}_{00}}(t_1,t_2; M)=\sup_{|Q_i|\geq t_i}\frac{1}{|Q_1||Q_2|}\left|\int_{Q_2}\int_{Q_1}f_{00}(x_1,x_2)dx_1dx_2\right|.
\end{equation*}

Let $t_1>\tau_1$, $t_2>\tau_2$, then considering \eqref{lf00_2}, we get
\begin{equation*}
\begin{split}
{\bar{f}_{00}}(t_1,t_2; M)=\sup_{|Q_i|\geq t_i}\frac{1}{|Q_1||Q_2|}\left|\int_{Q_2}\int_{Q_1}f_{00}(x_1,x_2)dx_1dx_2\right|\leq\\
\leq\sup_{{}^{s_1\geq t_1}_{s_2\geq t_2}}64\frac{\tau_1}{s_1}\frac{\tau_2}{s_2}\bar{f}\left(\tau_1,\tau_2; M\right)\leq64\frac{\tau_1}{t_1}\frac{\tau_2}{t_2}\bar{f}\left(\tau_1,\tau_2; M\right).
\end{split}
\end{equation*}

Let $t_1>\tau_1$, $t_2\leq\tau_2$, there are two possible cases: $|Q_1|=s_1>\tau_1$, $|Q_2|=s_2>\tau_2$ and $s_1>\tau_1$, $t_2<s_2\leq\tau_2$.

If $s_1>\tau_1$, $s_2>\tau_2$, then we use the estimate \eqref{lf00_2} and considering that $t_2\leq\tau_2$, we have
$$
\frac{1}{|Q_1||Q_2|}\left|\int_{Q_2}\int_{Q_1}f_{00}(x_1,x_2)dx_1dx_2\right|\leq64\frac{\tau_1}{s_1}\frac{\tau_2}{s_2}\bar{f}\left(\tau_1,\tau_2; M\right)\leq64\frac{\tau_1}{t_1}\bar{f}\left(\tau_1,t_2; M\right).
$$

If $s_1>\tau_1$, $s_2<\tau_2$, then 
$$
\frac{1}{|Q_1||Q_2|}\left|\int_{Q_2}\int_{Q_1}f_{00}(x_1,x_2)dx_1dx_2\right|\leq16\frac{\tau_1}{s_1}\bar{f}\left(\tau_1,s_2; M\right)\leq16\frac{\tau_1}{t_1}\bar{f}\left(\tau_1,t_2; M\right).
$$

Thus,
\begin{equation*}
{\bar{f}_{00}}(t_1,t_2; M)\leq64\frac{\tau_1}{t_1}\bar{f}\left(\tau_1,t_2; M\right).
\end{equation*}

Similarly we get an estimate
\begin{equation*}
{\bar{f}_{00}}(t_1,t_2; M)\leq64\frac{\tau_2}{t_2}\bar{f}\left(t_1,\tau_2; M\right),
\end{equation*}
at $t_1\leq\tau_1$, $t_2>\tau_2$.

At $t_1\leq\tau_1$, $t_2\leq\tau_2$ 4 cases are possible: $\begin{cases} s_1>\tau_1\\
s_2>\tau_2\end{cases}$, $\begin{cases} s_1>\tau_1\\
s_2<\tau_2\end{cases}$, $\begin{cases} s_1<\tau_1\\
s_2>\tau_2\end{cases}$, $\begin{cases} s_1<\tau_1\\
s_2<\tau_2\end{cases}$.

In the first case, we use the first relation from \eqref{lf00_2}, in the second - the first and second relations from \eqref{lf00_2}, in the third - the first and third, and in the fourth - all relations from \eqref{lf00_2}, then we have
$$
{\bar{f}_{00}}(t_1,t_2; M)\leq64\bar{f}\left(t_1,t_2; M\right).
$$
\end{proof}

\begin{lemma}\label{lf01}
Let $G_\tau$- partitioning $\mathbb{R}^2$ into rectangles, $f(x_1,x_2)$ - locally integrable on $\mathbb{R}^2$ and function $f_{01}$, $f_{10}$ are defined by the equalities \eqref{f01} and \eqref{f10} respectively. Then
\begin{equation}\label{lf01_1}
\bar{f}_{01}(t_1,t_2; M)\leq\left\lbrace
\begin{aligned}
    & 8\frac{\tau_1}{t_1}\left[3\bar{f}\left(\tau_1,t_2; M\right)+4\frac{\tau_2}{t_2}\bar{f}\left(\tau_1,\tau_2; M\right)\right],\;\;\;  t_1>\tau_1, t_2>\tau_2  \\
    & 56\frac{\tau_1}{t_1}\bar{f}\left(\tau_1,\tau_2; M\right),\;\;\;\;\;\;\;\;\;\; t_1>\tau_1, t_2\leq\tau_2  \\
    & 8\left[3\bar{f}\left(t_1,t_2; M\right)+4\frac{\tau_2}{t_2}\bar{f}(t_1,\tau_2; M)\right],\;\;\;\;\;\; t_1\leq\tau_1, t_2>\tau_2  \\
    & 56\bar{f}(t_1,\tau_2; M),\;\;\;\;\;\;\;\;\;\;\;\;\;\;\;\;\;  t_1\leq\tau_1, t_2\leq\tau_2  \\
\end{aligned}
\right.,
\end{equation}
$$
\bar{f}_{10}(t_1,t_2; M)\leq\left\lbrace
\begin{aligned}
    & 8\frac{\tau_2}{t_2}\left[3\bar{f}\left(t_1,\tau_2; M\right)+4\frac{\tau_1}{t_1}\bar{f}\left(\tau_1,\tau_2; M\right)\right],\;\;\;  t_1>\tau_1, t_2>\tau_2  \\
    & 8\left[3\bar{f}\left(t_1,t_2; M\right)+4\frac{\tau_1}{t_1}\bar{f}(\tau_1,t_2; M)\right],\;\;\;\;\;\;\;\;\;\;  t_1>\tau_1, t_2\leq\tau_2  \\
    & 56\frac{\tau_2}{t_2}\bar{f}\left(\tau_1,\tau_2; M\right),\;\;\;\;\;\;  t_1\leq\tau_1, t_2>\tau_2  \\
    & 56\bar{f}(\tau_1,t_2; M),\;\;\;\;\;\;\;\;\;\;\;\;\;\;\;\;\;  t_1\leq\tau_1, t_2\leq\tau_2 \\
\end{aligned}
\right..
$$
\end{lemma}

\begin{proof}

Let $|Q_1|=s_1$, $|Q_2|=s_2$. Let us prove the inequality
\begin{equation}\label{lf01_2}
\frac{1}{|Q_1||Q_2|}\left|\int_{Q_2}\int_{Q_1}f_{01}(x_1,x_2)dx_1dx_2\right|\leq\left\lbrace
\begin{aligned}
    & 8\frac{\tau_1}{s_1}\left[3\bar{f}\left(\tau_1,s_2; M\right)+4\frac{\tau_2}{s_2}\bar{f}\left(\tau_1,\tau_2; M\right)\right],\;\;\;  s_1>\tau_1, s_2>\tau_2  \\
    & 8\frac{\tau_1}{s_1}\bar{f}\left(\tau_1,\tau_2; M\right),\;\;\;\;\;\;\;\;\;\; s_1>\tau_1, s_2\leq\tau_2  \\
    & 2\left[3\bar{f}\left(t_1,t_2; M\right)+4\frac{\tau_2}{t_2}\bar{f}(t_1,\tau_2; M)\right],\;\;\;\;\;\; s_1\leq\tau_1, s_2>\tau_2  \\
    & 2\bar{f}(s_1,\tau_2; M),\;\;\;\;\;\;\;\;\;\;\;\;\;\;\;\;\;  s_1\leq\tau_1, s_2\leq\tau_2  \\
\end{aligned}
\right..
\end{equation}

Consider the case when $|Q_1|=s_1\leq\tau_1$, $|Q_2|=s_2\leq\tau_2$. Let's use the relation \eqref{prf00}, where there are terms we need, and apply their corresponding estimates in \eqref{ocf00}:
\begin{equation}\label{star_seven}
\begin{split}
&\frac{1}{|Q_1||Q_2|}\left|\int_{Q_2}\int_{Q_1}f_{01}(x_1,x_2)dx_1dx_2\right|\leq\left|\frac{1}{|Q_1|}\sum_{|{I_m^2}\cap Q_2|>0}\frac{|{I_m^2}\cap Q_2|}{|I_m^2||Q_2|}\int_{Q_1}\int_{I_m^2}f(x_1,x_2)dx_2dx_1\right|+\\
&+\left|\sum_{|{I_k^1}\cap Q_1|>0}\sum_{|{I_m^2}\cap Q_2|>0}\frac{|{I_k^1}\cap Q_1||{I_m^2}\cap Q_2|}{|I_k^1||I_m^2||Q_1||Q_2|}\int_{I_k^1}\int_{I_m^2}f(x_1,x_2)dx_2dx_1\right|\leq\bar{f}(s_1,\tau_2; M)+\bar{f}(\tau_1,\tau_2; M).
\end{split}
\end{equation}

Then we get
$$
\frac{1}{|Q_1||Q_2|}\left|\int_{Q_2}\int_{Q_1}f_{01}(x_1,x_2)dx_1dx_2\right|\leq2\bar{f}(s_1,\tau_2; M).
$$

When $|Q_1|=s_1>\tau_1$, $|Q_2|=s_2\leq\tau_2$ according to the lemma \ref{l3}, we have 
$$
\int_{I_k^1}f_{01}(x_1,x_2)dx_1=0.
$$

Using lemma \ref{l1}, we obtain 
\begin{equation*}
\begin{split}
&\frac{1}{|Q_1||Q_2|}\left|\int_{Q_2}\int_{Q_1}f_{01}(x_1,x_2)dx_1dx_2\right|\leq\frac{1}{|Q_1||Q_2|}\left(\left|\int_{Q_2}\int_{Q_1'} f_{01}(x_1,x_2)dx_1dx_2\right|+\right.\\
&\left.+\left|\int_{Q_2}\int_{Q_1''} f_{01}(x_1,x_2)dx_1dx_2\right|\right)=\frac{|{Q_1}'|}{|Q_1|}\frac{1}{{|Q_1}'||Q_2|}\left|\int_{Q_2}\int_{Q_1'} f_{01}(x_1,x_2)dx_1dx_2\right|+\\
&+\frac{|Q_1''|}{|Q_1|}\frac{1}{|Q_1''||Q_2|}\left|\int_{Q_2}\int_{Q_1''} f_{01}(x_1,x_2)dx_1dx_2\right|.
\end{split}
\end{equation*}

Applying to each term the relation \eqref{star_seven}, we optain
\begin{equation*}
\frac{1}{|Q_1||Q_2|}\left|\int_{Q_2}\int_{Q_1}f_{01}(x_1,x_2)dx_1dx_2\right|\leq\frac{2\tau_1}{s_1}\left[2\bar{f}(|Q_1'|, \tau_2; M)+2\bar{f}(|Q_1''|, \tau_2; M)\right]\leq8\frac{\tau_1}{s_1}\bar{f}\left(\tau_1,\tau_2; M\right),
\end{equation*}
where $\tau_1\leq|Q_1'|\leq2\tau_1$, $\tau_1\leq|Q_1''|\leq2\tau_1$.

In the case, when $|Q_1|=s_1\leq\tau_1$, $|Q_2|=s_2>\tau_2$, we apply lemma \ref{l2}, then
\begin{equation*}
\begin{split}
&\frac{1}{|Q_1||Q_2|}\left|\int_{Q_2}\int_{Q_1}f_{01}(x_1,x_2)dx_1dx_2\right|\leq\frac{1}{|Q_1||Q_2|}\left(\left|\int_{Q_2'}\int_{Q_1} f_{01}(x_1,x_2)dx_1dx_2\right|+\right.\\
&\left.+\left|\int_{Q_2''}\int_{Q_1} f_{01}(x_1,x_2)dx_1dx_2\right|+\left|\int_{Q_2'''}\int_{Q_1} f_{01}(x_1,x_2)dx_1dx_2\right|\right).
\end{split}
\end{equation*}

Let's estimate the first two terms:
\begin{equation*}
\begin{split}
&\frac{1}{|Q_1||Q_2|}\left(\left|\int_{Q_2'}\int_{Q_1} f_{01}(x_1,x_2)dx_1dx_2\right|+\left|\int_{Q_2''}\int_{Q_1} f_{01}(x_1,x_2)dx_1dx_2\right|\right)\leq\\
&\leq4\frac{\tau_2}{s_2}\bar{f}(s_1,\tau_2; M)+4\frac{\tau_2}{s_2}\bar{f}(s_1,\tau_2; M)=8\frac{\tau_2}{s_2}\bar{f}(s_1,\tau_2; M).
\end{split}
\end{equation*}

Estimating the third term, we get
\begin{equation*}
\begin{split}
&\frac{1}{|Q_1||Q_2|}\left|\int_{Q_2'''}\int_{Q_1} f_{01}(x_1,x_2)dx_1dx_2\right|\leq\frac{1}{|Q_1||Q_2|}\left|\int_{Q_2'''}\int_{Q_1} f(x_1,x_2)dx_1dx_2\right|+\\
&+\left|\sum_{|{I_k^1}\cap Q_1|>0}\frac{|I_k^1\cap Q_1|}{|Q_1||Q_2||I_k^1|}\int_{Q_2'''}\int_{I_k^1} f(x_1,x_2)dx_1dx_2\right|\leq\\
&\leq3\bar{f}\left(s_1,s_2; M\right)+3\bar{f}\left(\tau_1, s_2; M\right)\leq6\bar{f}\left(s_1,s_2; M\right).
\end{split}
\end{equation*}

Summing up the results, we get:
\begin{equation}\label{f01_l5}
\frac{1}{|Q_1||Q_2|}\left|\int_{Q_2}\int_{Q_1} f_{01}(x_1,x_2)dx_1dx_2\right|\leq2\left[3\bar{f}\left(s_1,s_2; M\right)+4\frac{\tau_2}{s_2}\bar{f}(s_1,\tau_2; M)\right].
\end{equation}

In the case, when  $|Q_1|=s_1>\tau_1$, $|Q_2|=s_2>\tau_2$, we apply lemma \ref{l1}, then
\begin{equation*}
\begin{split}
&\frac{1}{|Q_1||Q_2|}\left|\int_{Q_2}\int_{Q_1}f_{01}(x_1,x_2)dx_1dx_2\right|\leq\frac{1}{|Q_1||Q_2|}\left(\left|\int_{Q_2}\int_{Q_1'} f_{01}(x_1,x_2)dx_1dx_2\right|+\right.\\
&\left.+\left|\int_{Q_2}\int_{Q_1''} f_{01}(x_1,x_2)dx_1dx_2\right|\right)=\frac{|{Q_1}'|}{|Q_1|}\frac{1}{|Q_1'||Q_2|}\left|\int_{Q_2}\int_{Q_1'} f_{01}(x_1,x_2)dx_1dx_2\right|+\\
&+\frac{|Q_1''|}{|Q_1|}\frac{1}{|Q_1''||Q_2|}\left|\int_{Q_2}\int_{Q_1''} f_{01}(x_1,x_2)dx_1dx_2\right|.
\end{split}
\end{equation*}

Applying the estimate \eqref{f01_l5} at $|Q_1|=s_1\leq\tau_1$, $|Q_2|=s_2>\tau_2$, we get
$$
\frac{1}{|Q_1||Q_2|}\left|\int_{Q_2}\int_{Q_1}f_{01}(x_1,x_2)dx_1dx_2\right|\leq8\frac{\tau_1}{s_1}\left[3\bar{f}\left(\tau_1, s_2; M\right)+4\frac{\tau_2}{s_2}\bar{f}\left(\tau_1,\tau_2; M\right)\right].
$$
$$
\bar{f}_{01}(t_1,t_2; M)\leq8\frac{\tau_1}{t_1}\left[3\bar{f}\left(\tau_1,t_2; M\right)+4\frac{\tau_2}{t_2}\bar{f}\left(\tau_1,\tau_2; M\right)\right].
$$

The proof of \eqref{lf01_1} follows from \eqref{lf01_2} just as the estimate \eqref{lf00_1} follows from \eqref{lf00_2} in lemma  \ref{lf00}.
Due to symmetry, the function estimate $\bar{f}_{10}$ is obtained similarly to the estimate for the function $\bar{f}_{01}$.
\end{proof}

\begin{lemma}\label{lf11}
Let $G_\tau$ is partitioning $\mathbb{R}^2$ into rectangles, $f(x_1,x_2)$  is locally integrable on $\mathbb{R}^2$ and function $f_{11}$ is defined by  \eqref{f11}. Then
$$
\bar{f}_{11}(t_1,t_2; M)\leq4\bar{f}(max(t_1,\tau_1), max(t_2,\tau_2)).
$$
\end{lemma}
\begin{proof}
\begin{equation*}
\begin{split}
&I=\frac{1}{|Q_1||Q_2|}\left|\int_{Q_2}\int_{Q_1} f_{11}(x_1,x_2)dx_1dx_2\right|=\\
&=\left|\sum_{I_k^1\cap{Q_1}\neq\oslash}\sum_{I_m^2\cap{Q_2}\neq\oslash}\frac{|I_m^2\cap{Q_2}||I_k^1\cap{Q_1}|}{|I_k^1||I_m^2||Q_1||Q_2|}\int_{I_m^2}\int_{I_k^1}f(x_1',x_2')dx_1'dx_2'\right|.
\end{split}
\end{equation*}

As $|I_m^2\cap{Q_2}|\leq\min(s_2,\tau_2)$ and $|I_k^1\cap{Q_1}|\leq\min(s_1,\tau_1)$, then 
$$
\frac{|I_m^2\cap{Q_2}||I_k^1\cap{Q_1}|}{\tau_1\cdot\tau_2\cdot s_1\cdot s_2}\leq\frac{1}{\max(\tau_1,s_1)\max(\tau_2,s_2)}.
$$

Then
\begin{equation*}
\begin{split}
&I=\frac{1}{|Q_1||Q_2|}\left|\int_{Q_2}\int_{Q_1} f_{11}(x_1,x_2)dx_1dx_2\right|\leq\\
&\leq \frac{1}{\max(\tau_1,s_1)\max(\tau_2,s_2)}\left|\sum_{I_k^1\cap{Q_1}\neq\oslash}\sum_{I_m^2\cap{Q_2}\neq\oslash}\int_{I_m^2}\int_{I_k^1} f(x_1',x_2')dx_1'dx_2'\right|=\\
&=\frac{1}{\max(\tau_1,s_1)\max(\tau_2,s_2)}\left|\int_{\tilde{Q_2}}\int_{\tilde{Q_1}} f(x_1',x_2')dx_1'dx_2'\right|,  
\end{split}
\end{equation*}
here $\tilde{Q_1}=\cup_{I_k^1\cap{Q_1}\neq\oslash}I_k^1$, $\tilde{Q_2}=\cup_{I_m^2\cap{Q_2}\neq\oslash}I_m^2$ - segments, and such that $\max(s_i,\tau_i)\leq|\tilde{Q_i}|\leq2\max(s_i,\tau_i)$.
Consequently
$$
|I|\leq\frac{|\tilde{Q_1}||\tilde{Q_2}|}{\max(\tau_1,s_1)\max(\tau_2,s_2)}\bar{f}(\max(\tau_1,s_1),\max(\tau_2,s_2))\leq4\bar{f}(\max(\tau_1,s_1),\max(\tau_2,s_2)).
$$

Further,
\begin{equation*}
\begin{split}
&\bar{f}_{11}(t_1,t_2; M)=\sup_{|Q_i|\geq t_i}\frac{1}{|Q_1||Q_2|}\left|\int_{Q_2}\int_{Q_1} f_{11}(x_1,x_2)dx_1dx_2\right|\leq\\
&\leq4\bar{f}(\max(\tau_1,|Q_1|),\max(\tau_2,|Q_2|))\leq4\bar{f}(\max(\tau_1,t_1)\max(\tau_2,t_2)).
\end{split}
\end{equation*}
\end{proof}

We will use the classic Hardy inequalities. Let us formulate them in the form of a lemma.

\begin{lemma}[{Hardy's inequality}]\label{Hardy}
Let $1\leq q<\infty$, $\alpha>0$, then the inequalities hold
$$
\left(\int_0^\infty\left(t^\alpha\int_t^\infty \varphi(s)ds\right)^q\frac{dt}{t}\right)^\frac{1}{q}\leq\alpha^{-1}\left(\int_0^\infty\left(t^{1+\alpha}\varphi(t)\right)^q\frac{dt}{t}\right)^\frac{1}{q},
$$
$$
\left(\int_0^\infty\left(t^{-\alpha}\int_0^t \varphi(s)ds\right)^q\frac{dt}{t}\right)^\frac{1}{q}\leq\alpha^{-1}\left(\int_0^\infty\left(t^{1-\alpha}\varphi(t)\right)^q\frac{dt}{t}\right)^\frac{1}{q}.
$$
\end{lemma}

\section{ The main result}

Consider the interpolation method for anisotropic spaces proposed by E.D. Nursultanov \cite{NED6}. This method is based on ideas from the work of G. Sparr \cite{Sparr}, D.L. Fernandez \cite{Fernandez1}-\cite{Fernandez3} and others  \cite{Cobus}, \cite{Krepkogorski1}, \cite{Krepkogorski2}.

Let ${\bf A_0} = (A_1^{0},A_2^{0}),\; {\bf A_1} = (A_1^1,A_2^1) $ two anisotropic spaces, $ E =
\{\varepsilon=(\varepsilon_1,\varepsilon_2):\varepsilon_i=0,$ or $\varepsilon_i=1,\;\;i=1,2 \}.$ For arbitrary $\varepsilon\in E$ define space ${\bf{A}_\varepsilon} = (A_1^{\varepsilon_1},A_2^{\varepsilon_2})$ with the norm
$$
\|a\|_{\bf A_\varepsilon} = \|\|a\|_{A_1^{\varepsilon_1}}\|_{A_2^{\varepsilon_2}}.
$$

Let $0<\bar\theta=(\theta_1,\theta_2)< 1$, $0<\bar{q}=(q_1,q_2)\leq\infty$.
By ${\bf A}_{\bar{\theta},\bar{q}} = ({\bf A_0}, {\bf A_1})_{\bar{\theta},\bar{q}}$ 
denote the linear subset $\sum_{\varepsilon\in E}{\bf {A}}_\varepsilon$, for whose elements it is true:
$$
\|a\|_{{\bf A}_{\bar{\theta},\bar{q}}} =\left(\int^{\infty}_{0}\left(\int^{\infty}_{0}\left(t_1^{-\theta_1}t_2^{-\theta_2}K(t_1,t_2)\right)^{q_1}\frac{dt_1}{t_1}\right)^\frac{q_2}{q_1}\frac{dt_2}{t_2}\right)^\frac{1}{q_2}<\infty,
$$
where 
$$
K(t,a;{\bf A_0},{\bf A_1}) = \inf\{\sum_{\varepsilon\in E} t^{\varepsilon}
\|a_{\varepsilon}\|_{\bf A_\varepsilon} \;:\; a=\sum_{\varepsilon\in E}
a_\varepsilon,\; a_\varepsilon\in {\bf A}_\varepsilon\},
$$
where $ t^{\varepsilon} = t_1^{\varepsilon_1} t_2^{\varepsilon_2}$.

\begin{theorem}
Let $M$ is the set of all rectangles in $\mathbb{R}^2$,  $1<\bar{p}_0<\bar{p}_1<\infty$, $1\leq\bar{q}_0,\bar{q},\bar{q}_1\leq\infty$, $0<\bar{\theta}=(\theta_1,\theta_2)<1$ then
\begin{equation*}
(N_{\bar{p}_0,\bar{q}_0}(M), N_{\bar{p}_1,\bar{q}_1}(M))_{\bar{\theta},\bar{q}}=N_{\bar{p},\bar{q}}(M),
\end{equation*}
where $\frac{1}{\bar{p}}=\frac{1-\bar{\theta}}{\bar{p}_0}+\frac{\bar{\theta}}{\bar{p}_1}$.
\end{theorem}

\begin{remark}
Condition $1\leq\bar{q}_0,\bar{q},\bar{q}_1\leq\infty$ can be replaced by the condition $0<\bar{q}_0,\bar{q},\bar{q}_1\leq\infty$, the statement will remain true. This statement of the theorem is presented to shorten and simplify the proof.
\end{remark}

\begin{proof}
Let's prove the embedding
$$
N_{\bar{p},\bar{q}}(M) \hookrightarrow\left(N_{\bar{p}_0,\bar{q}_0}(M), N_{\bar{p}_1,\bar{q}_1}(M)\right)_{\bar{\theta},\bar{q}}.
$$

Let $f\in N_{\bar{p},\bar{q}}(M)$, $f_{00}$, $f_{01}$, $f_{10}$, $f_{11}$ defined by formulas \eqref{f01}-\eqref{f00}. Then, using lemmas  \ref{lf00}, \ref{lf01}, \ref{lf11}, we get
\begin{equation*}
\begin{split}
&\left\|f_{00}\right\|_{N_{(p_1^0,p_2^0),(1,1)}}= \int^{\infty}_{0}\int^{\infty}_{0}t_1^{\frac{1}{p_1^0}-1}t_2^{\frac{1}{p_2^0}-1}\bar{f}_{00}(t_1,t_2)dt_1dt_2\leq C_{00}\left(\int^{\tau_2}_{0}\int^{\tau_1}_{0}t_1^{\frac{1}{p_1^0}-1}t_2^{\frac{1}{p_2^0}-1}\bar{f}(t_1,t_2)dt_1dt_2+\right.\\
& \left.+\tau_1^{\frac{1}{p_1^0}}\int^{\tau_2}_{0}t_2^{\frac{1}{p_2^0}-1}\bar{f}\left(\tau_1,t_2)\right)dt_2+\tau_2^{\frac{1}{p_2^0}}\int^{\tau_1}_{0}t_1^{\frac{1}{p_1^0}-1}\bar{f}\left(t_1,\tau_2\right)dt_1+\tau_1^{\frac{1}{p_1^0}}\tau_2^{\frac{1}{p_2^0}}\bar{f}\left(\tau_1,\tau_2\right)\right);
\end{split}
\end{equation*}

\begin{equation*}
\begin{split}
&\left\|f_{01}\right\|_{N_{(p_1^0,p_2^1),(1,1)}}\leq C_{01}\left(\int^{\infty}_{\tau_2}\int^{\tau_1}_{0}t_1^{\frac{1}{p_1^0}-1}t_2^{\frac{1}{p_2^1}-1}\bar{f}(t_1,t_2)dt_1dt_2+\right.\\
&\left. +\tau_2^\frac{1}{p_2^1}\int^{\tau_1}_{0}t_1^{\frac{1}{p_1^0}-1}\bar{f}(t_1,\tau_2)dt_1+\tau_1^\frac{1}{p_1^0}\int^{\infty}_{\tau_2}t_2^{\frac{1}{p_2^1}-1}\bar{f}\left(\tau_1,t_2\right)dt_2+\tau_1^\frac{1}{p_1^0}\tau_2^\frac{1}{p_2^1}\bar{f}\left(\tau_1,\tau_2\right)\right);
\end{split}
\end{equation*}

\begin{equation*}
\begin{split}
&\left\|f_{10}\right\|_{N_{(p_1^1,p_2^0),(1,1)}}\leq C_{10}\left(\int^{\tau_2}_{0}\int^{\infty}_{\tau_1}t_1^{\frac{1}{p_1^1}-1}t_2^{\frac{1}{p_2^0}-1}\bar{f}(t_1,t_2)dt_1dt_2+\right.\\
&\left.+\tau_1^\frac{1}{p_1^1}\int^{\tau_2}_{0}t_2^{\frac{1}{p_2^0}-1}\bar{f}(\tau_1,t_2)dt_2+\tau_2^\frac{1}{p_2^0}\int^{\infty}_{\tau_1}t_1^{\frac{1}{p_1^1}-1}\bar{f}\left(t_1,\tau_2\right)dt_1+\tau_1^\frac{1}{p_1^1}\tau_2^\frac{1}{p_2^0}\bar{f}\left(\tau_1,\tau_2\right)\right);
\end{split}
\end{equation*}

\begin{equation*}
\begin{split}
&\left\|f_{11}\right\|_{N_{(p_1^1,p_2^1),(1,1)}}\leq C_{11}\left(\tau_1^\frac{1}{p_1^1}\tau_2^\frac{1}{p_2^1}\bar{f}(\tau_1,\tau_2)+\tau_2^\frac{1}{p_2^1}\int^{\infty}_{\tau_1} t_1^{\frac{1}{p_1^1}-1}\bar{f}(t_1,\tau_2)dt_1+\right.\\
&\left.+\int^{\infty}_{\tau_2}t_2^{\frac{1}{p_2^1}-1}\bar{f}(\tau_1,t_2)dt_2+\int^{\infty}_{\tau_2}\int^{\infty}_{\tau_1}t_1^{\frac{1}{p_1^1}-1}t_2^{\frac{1}{p_2^1}-1}\bar{f}(t_1,t_2)dt_1dt_2\right).
\end{split}
\end{equation*}
Then
\begin{equation*}
\begin{split}
&K(t_1,t_2, f)= K\left(t_1,t_2,f; N_{(p_1^0,p_2^0),(1,1)}, N_{(p_1^1,p_2^1),(1,1)}\right)\leq\\
& \leq\|f_{00}\|_{N_{(p_1^0,p_2^0),(1,1)}}+t_1\|f_{10}\|_{N_{(p_1^1,p_2^0),(1,1)}}+t_2\|f_{01}\|_{N_{(p_1^0,p_2^1),(1,1)}}+t_1t_2\|f_{11}\|_{N_{(p_1^1,p_2^1),(1,1)}}=\\
&=I_{00}+I_{01}+I_{10}+I_{11}.
\end{split}
\end{equation*}
$$
F(K)=\left(\int^{\infty}_{0}\left(\int^{\infty}_{0}\left(t_1^{-\theta_1}t_2^{-\theta_2}K(t_1,t_2)\right)^{q_1}\frac{dt_1}{t_1}\right)^\frac{q_2}{q_1}\frac{dt_2}{t_2}\right)^\frac{1}{q_2}\leq F(I_{00})+F(I_{01})+F(I_{10})+F(I_{11}).
$$

Next, we make a replacement
$$\tau_1=t_1^{\frac{1}{\frac{1}{p_1^0}-\frac{1}{p_1^1}}};\;\;\;\ \tau_2=t_2^{\frac{1}{\frac{1}{p_2^0}-\frac{1}{p_2^1}}}$$
and applying Hardy's inequality (see lemma \ref{Hardy}), we obtain 
\begin{equation*}
\begin{split}
F(K)=\left(\int^{\infty}_{0}\left(\int^{\infty}_{0}\left(t_1^{-\theta_1}t_2^{-\theta_2}K(t_1,t_2)\right)^{q_1}\frac{dt_1}{t_1}\right)^\frac{q_2}{q_1}\frac{dt_2}{t_2}\right)^\frac{1}{q_2}\leq C\left\|f\right\|_{N_{\bar{p},\bar{q}}(M)}.
\end{split}
\end{equation*}

Reverse nesting $(N_{\bar{p}_0,\bar{q}_0}(M), N_{\bar{p}_1, \bar{q}_1}(M))_{\bar{\theta}, \bar{q}}\hookrightarrow N_{\bar{p}, \bar{q}}(M)$ was proven in paper \cite{NED6} (see Theorem 1). 

\end{proof}

\renewcommand{\refname}{References}

\end{document}